\documentclass[11pt]{article}
\usepackage{fullpage}
\usepackage{amsthm}
\usepackage{enumitem}
\theoremstyle{plain}

\usepackage{latexsym}
\usepackage{amsmath}
\usepackage{amssymb}
\usepackage{amsthm}
\usepackage{amscd}
\usepackage{hyperref}
\usepackage{cite}
\usepackage{graphicx}
\usepackage{color}
\usepackage{amstext}
\usepackage{array}

\newtheorem{theorem}{Theorem}[subsection]
\newtheorem{theorem*}{Theorem}
\newtheorem{corollary}[theorem]{Corollary}
\newtheorem{lemma}[theorem]{Lemma}

\newtheorem{proposition}[theorem]{Proposition}
\newtheorem{definition}[theorem]{Definition}

\theoremstyle{definition}

\newtheorem{example}[theorem]{Example}
\newtheorem{remark}[theorem]{Remark}

\newcommand{\sym}{\mathrm{Sym}}

\newcommand{\eat}[1]{}

\let\hom\relax

\let\sym\relax
\let\lim\relax
\DeclareMathOperator{\hom}{Hom}

\DeclareMathOperator{\eendo}{End}
\DeclareMathOperator{\sym}{Sym}
\DeclareMathOperator{\Sch}{Sch}
\DeclareMathOperator{\XRepG}{Rep\Gamma}
\DeclareMathOperator{\xcolim}{colim}
\DeclareMathOperator{\xlim}{lim}
\DeclareMathOperator{\Hi}{H}

\newcommand{\RepG}{\text{Rep}\Gamma^{d}_{k}}

\newcommand{\intP}{\, \underline{\otimes} \,}
\newcommand{\nonrep}[2]{\underset{\Gamma^{d,#1} \rightarrow #2}{\xcolim}}

\newcommand{\higheri}[1]{\Hi^{#1}}

\begin{document}
\setcounter{page}{0}
\pagenumbering{arabic}

\title{Relating tensor structures on representations of general linear and symmetric groups}
\author{
Upendra Kulkarni\thanks{Chennai Mathematical Institute, Chennai, India.
({\tt upendra@cmi.ac.in}).
}
\and
 Shraddha Srivastava\thanks{Chennai Mathematical Institute, Chennai, India.
({\tt shraddha@cmi.ac.in})}
\and
K. V. Subrahmanyam \thanks{Chennai Mathematical Institute, Chennai, India.
({\tt kv@cmi.ac.in}).}
 }
\date{\today}

\maketitle


\begin{abstract}

For polynomial representations of $GL_n$ of a {\it fixed} degree, H.~Krause defined a new ``internal tensor product"
using the language of strict polynomial functors.
We show that over an arbitrary commutative base ring $k$, 
the Schur functor carries this internal tensor product to the usual Kronecker tensor product of symmetric group representations.
This is true even at the level of derived categories. 
The new tensor product is a substantial enrichment of the Kronecker tensor product.
E.g. in modular representation theory it brings in homological phenomena not visible on the symmetric group side.
We calculate the internal tensor product over any $k$ in several interesting cases involving classical functors and the Weyl functors.
We show an application to the Kronecker problem in characteristic zero when one partition has two rows or is a hook.

\end{abstract}


\section{Introduction}\label{sec:intro}

It has been recognized since the fundamental work of Issai Schur that
the representation theory of the symmetric group $S_d$ over a field $k$ is intimately connected with
the algebraic representation theory of the general linear group $GL_n(k)$.
This bridge has been used very fruitfully to investigate the representation theory of one of the two groups via knowledge of the other.  

\medskip

The passage from representations of $GL_{n}(k)$ to $kS_d$-modules is given by the
Schur functor (denoted by $\Sch$), via Schur-Weyl duality~\cite{Green80}. This duality between $GL_n(k)$ and $S_d$
is realized by the action of $GL_n(k) \times S_d$ on $({k^n})^{\otimes d}$.
Specifically, given a polynomial representation $M$ of $GL_n(k)$  of degree $d$, we obtain
the $kS_d$-module $\Sch(M) := \text{Hom}_{GL_n(k)} (({k^n})^{\otimes d}, M)$. 
We have the functor $\Sch$ for every pair $(n,d)$ of positive integers.
Under this functor, the tensor product of $GL_{n}(k)$-modules behaves as follows.
Given polynomial representations $M$ and $N$ of $GL_{n}(k)$ of degree $d$ and $e$ respectively, 
$M \otimes N$ is a polynomial representation of degree $d+e$. We have 
$$\Sch(M \otimes N) \simeq \text{Ind}_{S_d \times S_e}^{S_{d+e}} \Sch(M) \boxtimes \Sch(N).$$
A natural question arises: what corresponds to the (Kronecker) tensor product of $kS_d$-modules? 
To formulate this precisely, for the moment let $k$ be a field of characteristic 0 and $n \geq d$.
In this case $\Sch$ is well known to be an equivalence, 
so given polynomial representations $M$ and $N$ of $GL_n(k)$, both of degree $d$, we must have $\Sch(M \ ? \ N) \simeq \Sch(M) \otimes \Sch(N)$. 
Of course, $(M \ ? \ N)$ must be the polynomial representation whose formal character is the Kronecker product of the characters of $M$ and $N$. 
However, one would like to express the answer functorially in terms of $M$ and $N$. 
We show that the answer, which is valid even over an arbitrary commutative ring $k$, is $? = \intP$, 
an internal tensor product introduced recently by Krause in~\cite{Krause13}.

\medskip 

The operation $\intP$ and related prior ideas that inspired it have already proved pertinent in at least two other contexts:
Koszul duality~\cite{Chalupnik08, Touze13, Krause13} and derivatives of non-additive functors in algebraic topology, 
e.g. homology of certain Eilenberg-MacLane spaces~\cite{Touze13}.
Our results suggest that the ideas introduced by Cha\l{}upnik, Touz\'e and Krause should be valuable for representation theory of $S_d$ as well.

\medskip

The natural setting to define $\intP$ is the category of strict polynomial functors. 
It is well known that polynomial representations of $GL(V)$ such as symmetric and exterior powers of $V$
and more generally Schur and Weyl modules~\cite{ABW} are functorial in $V$. 
The category of strict polynomial functors of degree $d$ over a field $k$ was defined  by Friedlander and Suslin~\cite{FS97}. 
(The formulation requires care due to the difference between polynomials and polynomial functions over finite fields.)
They showed that for $n \geq d$,  this category is equivalent to the category $\text{Pol}_d(k^n)$ of polynomial representations
of degree $d$ of the general linear group scheme $GL({k^n})$ ~\cite[Theorem 3.2]{FS97}.
In his recent work on Koszul duality~\cite{Krause13} Krause generalized this result to an arbitrary commutative base ring 
by constructing the corresponding category $\RepG$ over any such ring $k$. 
(The notation $\RepG$ is explained in Section~\ref{sec:back} along with other background material. 
It is based on the fundamental role played by the divided power operation $\Gamma^d$ in defining this functor category.)
Krause also realized that as a functor category, $\RepG$ carries a tensor product $\intP$ via Day convolution. 
The Schur functor now takes the form $\Sch := \hom_{\RepG}(\otimes^d, -)$. 
We can now state our main results.

\medskip

(1) Our first basic result (Theorem~\ref{thm:main}) states that for strict polynomial functors $X,Y$ of degree $d$, there are natural isomorphisms 
$$\Sch(X \intP Y)\simeq\Sch(X)\otimes \Sch(Y)
\footnote
{While we were writing the first version of this paper, Aquilino and Reischuk uploaded a 
preprint~\cite{AR15} that proved this isomorphism and also calculated $\Gamma^\lambda \intP \Gamma^\mu$ via different arguments. 
Another independent proof of this isomorphism was given by Touz\'e~\cite{Touze15}. 
We thank these authors for discussion of their work with one of us (S.S.).
}  
\qquad {\rm and} \qquad
\Sch(\Bbb H (X,Y))\simeq\hom(\Sch(X), \Sch(Y)).$$
Here $\otimes$ and $\hom$ on the right hand side are the usual internal  
operations in the category of $kS_d$-modules and $\Bbb H$ is an internal Hom on strict polynomial functors. 
$\Bbb H$ satisfies the usual tensor-hom adjunction with respect to the internal tensor product $\intP$ 
and was in fact defined by Touz\'e~\cite{Touze14} before $\intP$. 
The Schur functor has both adjoints. We show in Corollary~\ref{cor:adjointinternal} that
Theorem~\ref{thm:main} formally implies a certain compatibility of each adjoint with one of the two internal structures, 
giving somewhat stronger versions of results obtained first in~\cite{Rei16} via different arguments.

\medskip

(2) Possibly more interesting than the above results is the fact that all of them continue to be valid at the derived level, 
as we show in Theorem~\ref{thm:intderived} and Corollary~\ref{cor:derivedadjointinternal}.
To see the relevance, let $k$ be a field of positive characteristic $p$. 
Now $\otimes$ and $\hom$ are exact for $kS_d$-modules. 
But the internal tensor $\intP$ and internal hom $\Bbb H$ are not and hence yield derived functors $\overset {\bf L} \intP$ and ${\bf R}\Bbb H$. 
Thus the closed monoidal structure on the symmetric group side lifts to a much richer structure on the general linear group side, 
the study of which is likely to be valuable for representation theory of either group. 
As an example of what one might get by such considerations, 
we observe in Corollary~\ref{cor:restricted} that for any $X, Y$ in $\RepG$ and $i \neq 0$ we have
$\Sch( H^i(X \overset{\bf L}\intP Y)) = 0$. 
Hence the highest weights of all composition factors of  such $H^i(X \overset{\bf L}\intP Y)$ must be non $p$-restricted.
In another direction, the Schur functor is always exact but its adjoints in general are not. 
It is well known that the derived functors of the adjoints of $\Sch$ contain valuable information in relating modular representation theories of $GL_n(k)$ and $S_d$~\cite{DEN}. 
Compatibility of these adjoints with the internal structure at the derived level further makes the case for relevance of
$\intP$ and $\Bbb H$ in modular representation theory. 

\medskip

We calculate $\intP$ for several classes of examples. 

\medskip

(3) {\it Classical exponential functors}. 
The first calculations involve the ``exponential" functors symmetric power ($\sym^d$), exterior power ($\wedge^d$) and divided power ($\Gamma^d$), 
the last being dual to $\sym^d$, see Section~\ref{subsec:div}. 
More generally, for a sequence $\lambda = \lambda_1, \ldots, \lambda_n$ of non-negative integers adding up to $d$, we have 
functors $\Gamma^\lambda := \Gamma^{\lambda_1} \otimes \ldots \otimes \Gamma^{\lambda_n}$ and likewise $\sym^\lambda$ and $\wedge^\lambda$.
Functors of the type $\Gamma^\lambda$ furnish a convenient projective generator for the category $\RepG$, 
so we first describe $X \intP \Gamma^\lambda$ for a general $X$ (Lemma~\ref{lm:paramet}). 
The answer is described as a weight space construction, reminiscent of the polarization operation in symbolic invariant theory.
This leads to a calculation of $\Gamma^\lambda \intP \Gamma^\mu$ (Proposition~\ref{prop:lambdamu}, also independently obtained in~\cite{AR15}).
It turns out that for any combination $A,B$ of exponential functors,  $A^\lambda \intP B^\mu$ is a direct sum of a single type of exponential functors 
over an indexing set defined combinatorially in terms of $\lambda$ and $\mu$ (Corollary~\ref{cor:symandwedge}). 
Notably the answers do not depend on the ground ring $k$ (except for the type $\wedge \intP \sym$ where it matters whether 2 is nonzero/nonunit in $k$,
see Remark~\ref{2problem}). 

\medskip

(4) {\it Weyl functors}. 
Weyl modules are universal highest weight modules in representation theory of $GL_n(k)$. 
The corresponding functors are  $\Delta(\mu)$, where $\mu$ is a partition of $d$.
We show that $\Delta(\mu) \intP \Gamma^{\lambda}$ has an explicitly described Weyl filtration (Proposition~\ref{prop:weyl}) that is independent of $k$.
Parallel results follow for dual Weyl functors using Koszul duality and for (dual) Specht modules using the Schur functor 
(Corollaries~\ref{cor:koz},~\ref{cor:nablasym} and~\ref{cor:specht}).
The internal tensor product of two Weyl functors $\Delta(\mu) \intP \Delta(\lambda)$ need not have a Weyl filtration, in contrast with their ordinary tensor product. 
But we show that their higher derived internal tensor products do vanish (Proposition~\ref{prop:twoweyl}). This is not true for two dual Weyl functors.

\medskip
 
(5) {\it Kronecker problem in characteristic $0$}. When $k$ is a field of characteristics $0$, calculating $\Delta(\mu) \intP \Delta(\lambda)$
is equivalent to the well known and difficult Kronecker problem. In Section~\ref{subsec:kron} we sketch an (impractical) algorithm 
described in terms of $\intP$. A more practical procedure can be devised when one of the partitions involved is a hook or 
has two rows and we spell out these cases. For the hook case we use the language of spin polynomial functors~\cite{Axt13}. 

\medskip

Our proofs consistently exploit the description of representations as functors, which allows the use of standard category theoretic tools such as the Yoneda lemma in various manifestations and limits/colimits. 
Parametrization of functors is another very useful idea, introduced in generality by Touz\'e. 
The equivalence of the category of strict polynomial functors with that of Schur algebra representations is also convenient, as it sometimes permits one to check facts by evaluating functors on a single object.
As in~\cite{Krause13}, we use unbounded derived categories for deriving functors.

\medskip

{\bf Acknowledgements}: 
S.S. was supported by a research fellowship from the National Board of Higher Mathematics. 
All three authors were partially supported by a grant from the Infosys foundation. 
During the course of this work,  S.S. had the opportunity to visit Henning Krause, Antoine Touz\'e and Wilberd van der Kallen. She thanks them for warm hospitality. 
She also thanks them, Karin Erdmann and Oded Yacobi for encouragment and fruitful discussions. 


\section{Strict polynomial functors and internal tensor product} \label{sec:back}

Unless mentioned otherwise, we will work over an arbitrary commutative ring $k$. 
Unadorned $\otimes$ and hom will denote these operations over $k$ and $^*$ will denote the $k$-linear dual. 
For a $k$-algebra $A$,  let $A$-Mod and Mod-$A$ respectively denote the categories of left and right $A$-modules. 

\medskip

We present here a streamlined development of relevant background, 
starting with polynomial representations of the general linear group scheme $GL(k^n)$ as motivation~(\ref{polyrep}). 
We formulate strict polynomial functors~(\ref{subsec:strict}) based on the divided power category~(\ref{subsec:div}) and 
recall important examples and constructions~(\ref{subsec:examples}) including the new internal 
hom and tensor structures due to Touz\'e and Krause~(\ref{subsec:inthom}). 
We list several consequences of the very useful Yoneda lemma that we need~(\ref{subsec:yoneda}). 
We recall the relation with the symmetric group $S_d$ on $d$ letters via the Schur functor~(\ref{subsec:schfun}). 
For completeness we have also included material on adjoints of the Schur functor~(\ref{subsec:schfun}) and on duality~(\ref{subsec:duality})  
even though these notions are used in only one or two places.
Our main reference is~\cite{Krause13}. 


\subsection{Polynomial representations and the Schur algebra}
\label{polyrep}

A representation $M$ of the $k$-group scheme $GL(k^n)$ is a comodule $M$ 
over the Hopf algebra $k[GL(k^n)] = \sym ((\eendo(k^n))^*) [det^{-1}]$, see~\cite[I.2.8]{Jantzen}.
We call $M$ a polynomial represenation (respectively, one of degree $d$) if the image of coaction map $\Delta_M: M \to M \otimes k[GL(k^n)]$ satisfies
\begin{center}
$\Delta_M(M) \subset M \otimes \sym (\eendo(k^n))^*$  (respectively, $\Delta_M(M) \subset M \otimes \sym^d (\eendo(k^n))^*$).
\end{center}
Thus the category $\text{Pol}_d(k^n)$ of polynomial representations of $GL_n(k)$ of degree $d$ is equivalent to 
the category of comodules over the coalgebra $\sym^d (\eendo(k^n))^*$ and hence to 
the category of modules over the dual algebra $(\sym^d (\eendo(k^n))^*)^*$. 
This is the Schur algebra $S_k(n,d)$, see~\cite{Green80}. 

\medskip

Restricting the action of $GL(k^n)$ on a representation $M$ to the diagonal subgroup scheme, 
we have the weight space decomposition of $M$~\cite[I.2.11]{Jantzen}.  
Clearly, each weight of an $S_k(n,d)$-module is a sequence $\lambda= \ (\lambda_{1},\ldots,\lambda_{n})$ 
of $n$ {\it non-negative} integers with $\sum \lambda_{i}= d$. 
We call such $\lambda$ the polynomial weights of degree $d$ and denote their set by $\Lambda(n, d)$. 
A weight $\lambda=(\lambda_1,\lambda_2,\ldots,\lambda_n)$ in $\Lambda(n, d)$ is dominant if it is a partition, i.e. 
if $\lambda_1 \geq \lambda_2 \geq \cdots \geq \lambda_n \geq 0$. 

\medskip

Corresponding to each partition $\lambda$, we have the Weyl module $\Delta(\lambda)$ and the dual Weyl module $\nabla(\lambda)$. 
These modules are free of finite rank over $k$ and their constructions are compatible with base change, so we omit the base ring $k$ in the notation. 
We have also suppressed the dependence on $n$ due to functoriality of these constructions discussed below in \ref{subsec:examples}.
The category $S_k(n,d)$-Mod has many good homological properties.  We recall below one that we need, see Proposition~\ref{prop:vanishing}.


\subsection{Divided power category}
\label{subsec:div}

Let $P_{k}$ denote the full subcategory of $k$-Mod with objects as finitely generated projective $k$-modules. 
Note that $P_{k}$ is closed under taking tensor products  and $k$-linear hom. 
Hence it is a closed monoidal category with the  usual hom-tensor adjunction
$$\hom_{P_{k}}(V \otimes W,  U)\simeq \hom_{P_{k}}(V,  \hom_{k}(W,  U)).$$

For $V$ in $P_k$, consider the $S_d$-action on $V^{\otimes{d}}$ by permuting the tensor factors.
The $d$-th divided power of $V$ is the module of invariants of this action, i.e., the $k$-module defined by  
\begin{center}
$\Gamma^{d}(V):= \ (V^{\otimes{d}})^{S_{d}}.$ 
\end{center}
For a group $G$ and a $kG$-module $M$, we have $(M^*)^G \simeq (M_G)^*$, where $M_G$ are the coinvariants.
For $G = S_d$ and $M = (V^*)^{\otimes d}$ with $V$ in $P_k$, we get $\Gamma^{d}(V) \simeq (M^*)^G \simeq (M_G)^*  = ( \sym^d(V^*))^*$.
Thus we have another interpretation of the Schur algebra:
$$S_k(n,d) \simeq  \Gamma^{d}(\text{Hom} (k^n, k^n)) \simeq \hom_{S_d} (({k^n})^{\otimes{d}}, ({k^n})^{\otimes{d}}).$$
This makes the following definition of a ``multi-object version" of the Schur algebra quite natural. 
The {\it divided power category}, $\Gamma^d{P_k}$, has the same objects as $P_k$ and morphisms defined by
\begin{center}
$\hom_{\Gamma^{d}P_{k}}(V, W) := \Gamma^{d}(\hom_{k}(V, W)) \simeq \hom_{S_d} (V^{\otimes{d}}, W^{\otimes{d}}).$
\end{center}
with the obvious composition using the last description. 
Thus $\Gamma^{d}P_{k}$ is equivalent to the full subcategory of $kS_d$-Mod 
consisting of objects of the form $V^{\otimes{d}}$ where $V$ is in $P_k$. 

 \medskip

Note that $\Gamma^{d}P_{k}$ inherits a duality and closed monoidal structure from the corresponding structures on 
$P_{k}$ or, equivalently, from $kS_d$-Mod.


\subsection{Strict polynomial functors}
\label{subsec:strict}

Strict polynomial functors over fields were introduced by Friedlander and Suslin~\cite{FS97}. 
An equivalent description based on divided powers was given by Pirashvili~\cite{Pirashvili}, 
with antecendents in earlier work by Bousfield~\cite{Bousfield} on homogeneous functors. 
This description was generalized to arbitrary commutative rings by Krause~\cite{Krause13} by using the category $\Gamma^d P_k$, see below. 
The idea of allowing finitely generated projective modules as arguments of polynomial functors seems to go back at least to~\cite{Kouwenhoven}.

\medskip

The category of strict polynomial functors of degree $d$, denoted by $\RepG$, is the category of $k$-linear representations of the category $\Gamma^d{P_k}$.
Thus objects of $\RepG$ are 
\begin{center}
$k$-linear covariant functors: $\Gamma^{d}P_{k} \ \to \ k$-Mod
\end{center}
and morphisms are the natural transformations between such functors. 
Being a functor category with an abelian target,  
$\RepG$ is also an abelian category in which (co)kernels, images, direct sums/products are calculated pointwise.  
As $k$-Mod has arbitrary (set-indexed) (co)products and hence (co)limits, so does $\RepG$.

\medskip

Since $\eendo_{\Gamma^{d}P_{k}} (k^n) = S_k(n,d)$,  we have the evaluation functor $ev_{k^{n}}: \RepG \to S_k(n,d)$-Mod taking $X \mapsto X(k^n)$. 
In fact, we have the following result~\cite[Theorem 3.2]{FS97},~\cite[Theorem 2.10]{Krause13}. 
The argument is sketched below in Section~\ref{subsec:yoneda}, item~\ref{item:FSproof}.

\begin{proposition}\label{thm:FS}
If $n \geq d$ then $ev_{k^{n}}:\XRepG^{d}_{k}\rightarrow S_{k}(n,d)${\rm-Mod}
is an equivalence of categories.
\end{proposition}
Thus study of $\RepG$ is equivalent to that polynomial representations of $GL_n$ for large $n$ 
and we will exploit this at times to transfer certain notions and results from representations to functors. 
Nonetheless, the functor point of view is very useful, for example, in defining parametrized functors
and an internal closed monoidal structure on $\RepG$. 

\medskip

In passing, consider the smallest cases $d=0$ and $d=1$. The category $\Gamma^{0}P_{k}$ is equivalent to a point with $\eendo(pt) = k$, so
$\XRepG^0_k$ is just $k$-Mod, considered as constant representations of $GL(V)$. $\XRepG^1_k$ is also equivalent to $k$-Mod, but for a different reason, namely $S_{k}(1,1) \simeq k$. 
More plainly, a $k$-linear functor $P_{k} \to k$-Mod may take any value on $k$ and this value determines the functor. 


\subsection{Important examples and constructions}
\label{subsec:examples}

\noindent {\bf Exponential functors.} Let $V$ be in $P_k$. 
The usual multilinear algebraic constructions $V^{\otimes{d}}$, $\wedge^{d}(V)$, $\sym^{d}(V)$ and $\Gamma^{d}(V)$ give functors in $\RepG$. 
The last three are degree $d$ components, respectively, of the exterior algebra $\wedge(V)$, 
symmetric algebra $\sym(V)$ and the divided power algebra $\Gamma(V)$, all of which are graded-(co)commutative bialgebras. 
Appropriate universal properties of these bialgebras give isomorphisms of graded-(co)commutative bialgebras
$$\Gamma{(V\oplus W)} \simeq \Gamma(V)\otimes\Gamma(W) \, ,\quad
\sym(V\oplus W)\simeq \sym(V)\otimes\sym(W) \, ,\quad
\wedge(V\oplus W)\simeq \wedge(V)\otimes \wedge(W).
$$

Applying $\wedge, \sym$ and $\Gamma$ to a number of copies of $V$ and 
taking the homogeneous component of a fixed multidegree gives strict polynomial functors. 
More explicitly, for $\lambda  \in \Lambda(n, d)$, we have strict polynomial functors
$$
\Gamma^{\lambda}:= \Gamma^{\lambda_{1}}  \otimes...\otimes \Gamma^{\lambda_{n}}  \, ,\qquad
\sym^{\lambda}:= \sym^{\lambda_{1}}\otimes \ldots \otimes \sym^{\lambda_{n}}  \, ,\qquad
\wedge^{\lambda}:= \wedge^{\lambda_{1}}\otimes\ldots\otimes \ \wedge^{\lambda_{n}}.
$$
Here $\otimes$ corresponds to the usual tensor product of representations, i.e. 
for functors $X$ and $Y$ of degree $d$ and  $e$ respectively, $(X \otimes Y)(V):=  X(V)\otimes Y(V)$ defines a functor of degree $d+e$. 

\medskip

\noindent {\bf Weyl and dual Weyl functors.} 
Morphisms among various $\Gamma^{\lambda}$, $\sym^{\lambda}$ and $\wedge^{\lambda}$ obtained via (co)multiplication on the respective bialgebras lead to more (in some sense all) functors. 
We note in particular the Weyl functors $\Delta(\lambda)$ and the dual Weyl functor $\nabla(\lambda)$, where $\lambda$ is a partition of $d$. 
These were defined by Akin, Buchsbaum and Weyman~\cite{ABW}.   
Evaluation takes (dual) Weyl functors to (dual) Weyl modules. (The dual Weyl functors were originally called Schur functors, 
but we will need that name for a completely different functor, one from $\RepG$ to $kS_d$-Mod, see~\ref{subsec:schfun}.)  
We recall here a basic vanishing result. 
\begin{proposition}\label{prop:vanishing}
${\rm Ext}^i_{\XRepG^d_k} (\Delta(\lambda), \nabla(\mu)) = 0$ for all partitions $\lambda$, $\mu$ and all $i >0$.
\end{proposition}
This is well known in the equivalent category $\text{Pol}_d(k^d)$ when $k$ is a field, see~\cite[Propositions II.4.13, A.10]{Jantzen}. 
Then use base change (e.g.~\cite[Theorem 5.3]{AB88}) to deduce the vanishing first for $k = \Bbb Z$ and then for arbitrary commutative $k$.

\medskip

\noindent {\bf Parametrized functors.} 
For $X$ in $\RepG$ and $V$ in $P_k$, define as in~\cite{Touze14} 
$$ X_V:=X(V \otimes -) \quad \text{and} \quad X^{V}:= X(\hom_{k}(V,-)) = X(V^* \otimes -) = X_{V^*}.$$
Clearly $X \mapsto X^V$ is exact and preserves (co)products and hence (co)limits.
The important special case $(\Gamma^d)^V$ was already considered in~\cite{FS97}, 
where the slightly different notation $\Gamma^{d,V}$ was used and which we will follow. 
Note that $\Gamma^{d,V}$ are precisely the representable functors in $\RepG$, as
$$\Gamma^{d,V} := \Gamma^{d}(\hom_{k}(V, -)) = \hom_{\Gamma^{d}P_{k}}(V, -).$$ 


\subsection {Consequences of the Yoneda lemma.}  
\label{subsec:yoneda}

Recall the very useful natural isomorphism $\hom_{\RepG} (\Gamma^{d,V}, X) \simeq X(V).$ 
We list below several consequences and related facts.
\begin{enumerate}[leftmargin=*]

\item The functor $\hom_{\RepG}(\Gamma^{d,V}, -)$ is exact and preserves arbitrary (co)products and (co)limits. 

\item The class $\{\Gamma^{d,V}\}$ forms a projective generator in $\RepG$. Moreover, 
for any $X$ in $\RepG$ we have~\cite[III.7, Theorem 1]{MacLaneS}
$$X = \nonrep{V}{X}\Gamma^{d,V},$$ 
where $V$ runs over the objects in $\Gamma^d {P_k}$ and the colimit is taken over the category of morphisms $\Gamma^{d,V} \to X$. 
This gives a useful method to prove natural isomorphisms between functors from $\RepG$, 
namely do it first for $\{\Gamma^{d,V} \}$ and then check behavior on colimits. 
E.g., we have~\cite[Lemma 4.1]{Touze14} $\hom_{\RepG}(X^W,Y) \simeq \hom_{\RepG}(X,Y_W)$. 
For $X=\Gamma^{d,V}$ this just says 
$Y(W \otimes V) \simeq Y_W(V)$. 
Now observe that each procedure turns colimits in the $X$ variable into limits.

\item\label{item:wtspace}Using $\Gamma{(V\oplus W)} \simeq \Gamma(V)\otimes\Gamma(W)$, we get 
$\Gamma^{d,k^n}\simeq  \underset{\lambda\in \Lambda(n, d)}\bigoplus \Gamma^{\lambda}.$
One checks that under the Yoneda isomorphism, the summand $\hom_{\RepG} (\Gamma^{\lambda}, X)$ of 
$\hom_{\RepG} (\Gamma^{d,k^n}, X)$ identifies with $X(k^n)_\lambda$, the $\lambda$-weight space of the 
$S_k(n,d)$-module $X(k^n)$, e.g.~\cite[Corollary 2.12]{FS97} or~\cite[Section 2]{AB88}.

\item \label{item:FSproof}{\it Proof of Proposition~\ref{thm:FS}}: Given $V$ in $P_k$ we have $V \overset {\oplus} \hookrightarrow$ some $k^n$, 
so $\Gamma^{d,V} \overset {\oplus} \hookrightarrow \Gamma^{d,k^n}$ and 
thus $\{\Gamma^{d,k^n} | n \in \Bbb N\}$ is a projective generator for $\RepG$. 
Recall the decomposition of $\Gamma^{d,k^n}$ in item \ref{item:wtspace}. 
As $n$ increases, new summands $\Gamma^{\lambda}$ appear in this decomposition only up until $n=d$. 
The last one to appear is $\lambda = (1, \ldots, 1)$ consisting of $d$ 1's.  
For all $n \geq d$, the same set of $\Gamma^{\lambda}$ appears, only with higher multiplicities. 
Thus for any single $n\geq d$, the functor $\Gamma^{d,k^n}$ is a projective generator. 

\item  \label{finite-criterion} 
Recall from~\cite{Krause13} the full subcategory rep$\Gamma^d_k$ of $\RepG$ 
consisting of $X$ for which $X(V)$ is in $P_k$ for all $V$. 
For this it is enough for $X(k^d)$ to be in $P_k$ by~\cite[Remark 2.11]{Krause13} . 
We amplify this slightly. 
First, if $X(k^d)$ is a projective $k$-module, then so is $X(V)$ for all $V$ in $P_k$, because
$$X(V) \simeq \hom_{\RepG}(\Gamma^{d,V}, X) \overset {\oplus} \hookrightarrow \hom_{\RepG}(\Gamma^{d,k^m}, X) 
\overset {\oplus} \hookrightarrow \hom_{\RepG}( \underset{\rm finite} \oplus \Gamma^{d,k^d}, X) \simeq \underset{\rm finite} \oplus X(k^d).$$
For finite generation, we again have equivalent characterizations given by the following circular implications:
$X$ in $\RepG$ is finitely generated 
$\overset{\rm def} \Leftrightarrow$ $X$ is a quotient of a finite direct sum of representable functors $\Gamma^{d,W}$
$\Rightarrow$ each $X(V)$ is a finitely generated $k$-module
$\Rightarrow X(k^d)$ is a finitely generated $k$-module
$\Rightarrow X(k^d)$ is a finitely generated $S_k(d,d)$-module 
$\Rightarrow$ (by the equivalence in~\ref{thm:FS}) $X$ is a quotient of a finite direct sum of $\Gamma^{d,k^d}$,
i.e. $X$ is finitely generated. 

\end{enumerate}


\subsection{Internal tensor product and Internal Hom}
\label{subsec:inthom}
Following Krause~\cite{Krause13}, the closed monoidal structure on $\Gamma^{d}P_{k}$ leads 
via Kan extension to a closed monoidal structure on the functor category. 
Thus on $\RepG$ we have an internal tensor product (also known as Day convolution) $\intP$ 
and an internal hom $\mathbb{H}$ with the following properties~\cite[Proposition 2.4]{Krause13}.
\begin{proposition}\label{prop:inthom}
The bifunctors $\mathbb{H}: {\XRepG^{d}_{k}}^{op} \times \XRepG^{d}_{k} \rightarrow \XRepG^{d}_{k}$ and 
$\intP : \XRepG^{d}_{k} \times \XRepG^{d}_{k} \rightarrow \XRepG^{d}_{k}$ are defined by
$$
X \intP Y := \nonrep{V}{X} \ \nonrep{W}{Y} \ \Gamma^{d,V\otimes W}
\qquad and \qquad 
\mathbb{H}(X, Y) := {\underset{\Gamma^{d,V} \rightarrow X}{\xlim}} \ \nonrep{W}{Y} \ \Gamma^{d,\hom_{k}(V, W)} .$$
There is a natural isomorphism 
 $$\hom_{\XRepG^{d}_{k}}(X\intP Y, Z)\simeq \hom_{\XRepG^{d}_{k}}(X,\mathbb{H}(Y,Z)).$$
\end{proposition}

It follows that we have $\Gamma^{d,V}\intP \ \Gamma^{d,W} \simeq \Gamma^{d,V\otimes W}$, 
$\mathbb{H}(\Gamma^{d,V}, \Gamma^{d,W}) \simeq \Gamma^{d,\hom_{k}(V, W)}$ and more generally for any $Y$ in $\RepG$
~\cite[Lemma 2.5]{Krause13} 
$$
\Gamma^{d,V} \intP \ Y \simeq Y^V \qquad, \qquad \mathbb{H}(\Gamma^{d,V}, Y) \simeq Y_V.
$$

Being a left adjoint, the internal tensor preserves colimits. 
In particular it is right exact and can be computed using a presentation by $\Gamma^{d,V}$. 
Similar remarks hold for $\Bbb H$. 
In addition there is a very useful alternative expression for $\Bbb H$, which was in fact the original definition by Touz\'e~\cite{Touze14}. 
We recover it as follows using compatibility of (co)limits with parametrization.
\begin{eqnarray*}
\mathbb{H}(X, Y)(U)&=&\Big( {\underset{\Gamma^{d,V} \rightarrow X}{\xlim}} \ \nonrep{W}{Y} \ \Gamma^{d,\hom_{k}(V, W)} \Big) (U)
\simeq {\underset{\Gamma^{d,V} \rightarrow X}{\xlim}} \ \nonrep{W}{Y} \ \hom_{\RepG} (\Gamma^{d,V \otimes U}, \ \Gamma^{d,W})\\
&\simeq& \hom_{\RepG} \Big(\nonrep{V}{X} \ \big(\Gamma^{d,V} \big)^U , \  \nonrep{W}{Y} \ \Gamma^{d,W} \Big) \\
&\simeq& \hom_{\RepG}(X^U,Y).
\end{eqnarray*}

The functors $\intP$ and $\Bbb H$ have derived versions, respectively $\overset{\bf L}\intP$ and ${\bf R} \Bbb H$, 
which may be calculated using appropriate projective/injective resolutions, see~\cite[Section 4]{Krause13} for details. 
The derived functor ${\bf R} \Bbb H(\wedge^d, -)$ was introduced as a version of Koszul duality 
by Cha\l{}upnik~\cite{Chalupnik08} and investigated further by Touz\'e~\cite{Touze13}. 
Touz\'e also calculated ${\bf R} \Bbb H$ involving classical exponential functors 
$\sym, \wedge, \Gamma$ and their Frobenius twists~\cite{Touze14}. 
Krause defined $\intP$ and used the derived functor of $\wedge^d \intP -$ to study Koszul/Ringel duality over arbitrary $k$~\cite{Krause13}.

\medskip

We note a simple fact for later use. Let $X$ and $Y$ be in $\RepG$. 
If $P \to X$ is a projective resolution, then so is $P^V \to X^V$, as $(-)^V$ is an exact functor and preserves projectivity. 
Therefore
\begin{equation}\label{rh}
{\bf R}\Bbb H (X,Y)(V) \simeq \Bbb H (P,Y)(V) \simeq \hom_{\RepG}(P^V,Y) \simeq {\bf R}{\rm Hom}_{\RepG}(X^V,Y).
\end{equation}


\subsection{Schur functor}
\label{subsec:schfun}

Classical Schur-Weyl duality arises from natural commuting actions on $(k^n)^{\otimes d}$ of $GL(k^n)$ on the left and $S_d$ on the right.
The centralizer of  the $S_d$-action is the Schur algebra $S = S_k(n,d)$ and, 
when $n > d$, the centralizer of $GL(k^n)$ is isomorphic to $(kS_d)^{op}$.
This leads to the Schur functor $\Sch$ and its adjoints, which we recall below. 
Mainly we will employ only two facts about $\Sch$: 
that it preserves all limits and colimits (as it has both adjoints) and the calculation of $\Sch(\Gamma^{d,V})$.
We will need the adjoints only once to prove two formal results (Corollaries~\ref{cor:adjointinternal},~\ref{cor:derivedadjointinternal}) about them. 

\medskip

To describe the set-up in terms of polynomial functors, we first recall the standard Schur algebra formulation. 
From item~\ref{item:wtspace} in Section~\ref{subsec:yoneda} we have the $S$-module decomposition
$S \simeq  \bigoplus \Gamma^{\lambda}(k^n)$ with $\lambda$ ranging over  $\Lambda(n, d)$.
Let $e$ be the idempotent in $S$ corresponding to the summand $\Gamma^{1^d}(k^n) = (k^n)^{\otimes d}$. 
We have the functor 
$$\Sch: S\text{-Mod} \to eSe\text{-Mod} \quad \text {given by} \quad M \mapsto eM \simeq eS \otimes_S M \simeq \hom_S(Se, M).$$
Using the usual tensor-hom adjunction, we get the left/right adjoints $\mathcal L$ and $\mathcal R$ of $\Sch$ given by standard formulas below.  
Using these, one easily sees further that $\Sch \mathcal L (N) \simeq N \simeq \Sch \mathcal R (N)$. 
$$\mathcal L (N) = Se \otimes _{eSe} N \quad \text{and} \quad \mathcal R (N) = \hom_{eSe} (eS, N).$$
In our case, the $S$-$eSe$ bimodule $Se \simeq (k^n)^{\otimes d}$. 
As $eSe \simeq \hom_S (Se, Se)^{op}$, we have $eSe \simeq kS_d$ by classical Schur-Weyl duality. 
One can see that the $eSe$-$S$ bimodule $eS \simeq ((k^n)^*)^{{\otimes d}}$. 
Combining with the equivalence~\ref{thm:FS}, we have the following formulation of the Schur functor and of its adjoints. 
Compare with the description in~\cite{Krause13} and its explanation in~\cite{Rei16}, see also~\cite[Section 3.1]{FP}.
$$\Sch:=  \hom_{\RepG}(\otimes^{d},-) : \RepG \rightarrow kS_{d}\text{-Mod},$$
$$\mathcal L(N): V \mapsto  V^{\otimes d} \otimes _{kS_d} N  \simeq (V^{\otimes d} \otimes N)_{S_d} 
\quad , \quad
\mathcal R (N): V \mapsto \hom_{kS_d} ( (V^{\otimes d})^*, N) \simeq (V^{\otimes d} \otimes N)^{S_d}.$$
For the description of the adjoints in terms of $S_d$-coinvariants and $S_d$-invariants 
one converts the natural right $S_d$-action on $V^{\otimes d}$ into a left action in the usual way.
Using these descriptions and the fact that both adjoints are one-sided inverses to $\Sch$, we immediately get
$$\mathcal L(W ^{\otimes d}) \simeq \sym^d_W  \quad ,  
\quad \mathcal R(W ^{\otimes d}) \simeq \Gamma^d_W \quad ,  
\quad \Sch (\sym^d_W) \simeq W ^{\otimes d} \simeq \Sch(\Gamma^d_W).$$
%
We will calculate $\Sch(\Gamma^{d,V})$ again in Proposition~\ref{prop:main} without using the formulas for adjoints. 

\begin{remark}\label{remark:adjointinternal} 
In passing we quote two results discovered recently by Reischuk, expressing the adjoints of $\Sch$ in terms of $\intP$ and $\Bbb H$.
We will not need these results, but include their proofs to highlight usefulness of the Yoneda lemma. 
See~\cite[Theorems 4.3 and 5.4]{Rei16} for original arguments.
$$
\mathcal L \Sch (X) \simeq X \intP \sym^d \qquad {\rm and} \qquad
\mathcal R \Sch (X) \simeq \Bbb H (\sym^d, X).$$   
$\mathcal L \Sch (-)$ and $- \intP \sym^d$ preserve colimits and take $\Gamma^{d,V}$ to $\sym^{d,V}$,
giving the first result. For the other, we have the following series of functorial isomorphisms, 
where we have suppressed $\hom_{\RepG}$ and $\hom_{kS_d}$ from the notation.
$$(Y, \mathcal R \Sch (X)) \simeq (\Sch (Y), \Sch (X)) \simeq (\mathcal L \Sch (Y), X) \simeq (Y \intP \sym^d, X) \simeq (Y, \Bbb H( \sym^d, X)).$$

\end{remark}

\begin{remark} \label{remark:schurequivalence}
When $d!$ is a unit in $k$ (e.g., when $k$ is a field of characteristic 0 or $p > d$), 
$\Sch$ is well known to be an equivalence of categories. 
This is because, for any $\lambda$ in $\Lambda(n, d)$, 
the appropriate symmetrization map $\otimes^{d}\rightarrow\Gamma^{\lambda}$ is surjective. 
Thus $\otimes^{d}$ itself is a projective generator for ${\RepG}$.  

\end{remark}


\subsection{Duality}  
\label{subsec:duality}

The contravariant dual $M^\circ$ of a polynomial representation $M$ is the $k$-linear dual $M^*$, 
made into a left $GL(k^n)$-module via matrix transpose~\cite[Section 2.7]{Green80}. 
We have the classical dualities $\Gamma^d(V)^* \simeq \sym^d(V^*)$ and $\wedge(V)^* \simeq \wedge(V^*)$, 
well known at least for $V$ free of finite rank. 
Using these,~\cite{ABW} gives more generally that $\Delta(\lambda)(V)^* \simeq \nabla(\lambda)(V^*)$ for such $V$. 

\medskip

Following~\cite[p.105]{Kouwenhoven} and~\cite[3.4]{Kuhn}, we define the contravariant dual (also known as Kuhn dual) 
$X^\circ$ of an arbitrary strict polynomial functor $X$ by  $X^\circ(V) := X(V^*)^*$.
Evaluation on $k^n$ takes this duality to the usual contravariant duality for $S_k(n,d)$-modules, justifying the notation. 
Thus ${\wedge^d}$ is self-dual, ${\Gamma^d}$ and $\sym^d$ are dual to each other 
and more generally so are $\Delta(\lambda)$ and $\nabla(\lambda)$.  

\medskip

Duality behaves well on rep$\Gamma^d_k$, i.e. on functors taking values in $P_k$. 
If $Y$ is in rep$\Gamma^d_k$ (equivalently, if $Y(k^d)$ is in $P_k$, see item~\ref{finite-criterion} in Section~\ref{subsec:yoneda}), 
so is $Y^\circ$. It follows for such $Y$ that $Y^{\circ\circ} \simeq Y$. 
This is false in general, e.g. consider $k = \Bbb Z$ and $Y$ taking torsion values.

\medskip

Duality is an exact functor if $k$ is self-injective (e.g., a field). 
For general $k$ one may use the derived functor $(-)^\diamond$ of $(-)^\circ$ introduced by Krause,
which we will need in Proposition~\ref{prop:twoweyl}.
By~\cite[Section 4]{Krause13},
$X^\diamond \simeq X^\circ$ for $X$ in rep$\Gamma^d_k$ as $(-)^\circ$ is exact on this Quillen exact category. 
Concretely,  a projective resolution $Q \to X \to 0$ stays exact after applying $(-)^\circ$, 
e.g. use that $Q(k^d) \to X(k^d) \to 0$ is split exact over $k$. 

\medskip

Following Krause but slightly more generally, we relate duality with $\Bbb H$ and $\intP$, 
compare~\cite[Lemmas 2.7, 2.8]{Krause13}.
We also include a part of~\cite[Lemma 4.5]{Krause13}, which we need in Proposition~\ref{prop:twoweyl}.
Note that the hypothesis on $X$ is valid when $X$ is a Weyl functor,  see~\cite[Section 4]{AB88}. 

\begin{lemma}\label{lemma:duality}
In $\XRepG^{d}_{k}$, $\Bbb H (X, Y^\circ) \simeq (X \intP Y)^\circ \simeq  \Bbb H (Y, X^\circ)$. 
Then $Y =\Gamma^d$ gives $X^\circ \simeq \Bbb H (X, \sym^d)$.  
In ${\rm D}(\XRepG^{d}_{k})$, we have 
$X \overset{\bf L}\intP Y^\diamond  \simeq  {\bf R}\mathbb{H}(X,Y)^{\diamond}$ when 
$X \simeq$ a bounded complex of finite direct sums of representable functors.
\end{lemma}

\begin{proof}
By symmetry of $\intP$, it suffices to show $\Bbb H (X, Y^\circ) \simeq (X \intP Y)^\circ$. 
As both expressions turn colimits in the $X$ variable into limits, 
it suffices to check this for $X = \Gamma^{d,W}$, which is immediate. 
The last statement follows by direct calculation using a projective resolution of $Y$. 
\end{proof}

In general one cannot move all occurrences of duality in~\ref{lemma:duality} to one side 
due to potentially bad behavior of duality outside rep$\Gamma^d_k$.  
If $Y$ is in rep$\Gamma^d_k$, then $\Bbb H(X,Y) \simeq \Bbb H(X,Y^\circ{}^\circ) \simeq (X \intP Y^\circ)^\circ$.
If $X \intP Y$ is in rep$\Gamma^d_k$, then $X \intP Y \simeq \Bbb H(X,Y^\circ)^\circ$, 
but for general $k$ this can fail even if $X$ and $Y$ are in rep$\Gamma^d_k$, see Remark~\ref{2problem}.


\section{Results}\label{sec:main}

Recall that unless otherwise mentioned, $k$ is an arbitrary commutative ring. 


\subsection{Sch is a closed monoidal functor}\label{sec:monoidal}

We prove that Sch takes the internal tensor product $\intP$ and internal hom $\Bbb H$ in $\RepG$ 
to the corresponding operations in $kS_d$-Mod. 
We show that this stays true at the derived level and note an interesting consequence for modular representations. 
We also recover and extend by purely formal means recent results of Reischuk~\cite{Rei16} 
about compatibility of each internal operation with an appropriate adjoint of the Schur functor.

\medskip

As  $\intP$ and $\Bbb H$ are defined using representable functors $\Gamma^{d,V}$, we first compute $\Sch(\Gamma^{d,V}).$

\begin{proposition}\label{prop:main}
$\Sch(\Gamma^{d,V})\simeq (V^*)^{\otimes{d}}$ with $S_{d}$ acting by permuting the factors of $(V^*)^{\otimes{d}}$ 
(left action corresponding to the natural right action on $V^{\otimes{d}}$). 
\end{proposition}

\begin{proof}
Let $e_1,e_2,\ldots,e_d$ be a basis of the vector space $k^d$. 
We identify $V^* \otimes k^d$ with $V_1 \oplus V_2 \oplus \cdots \oplus V_d$ where $V_i = V^* \otimes ke_i \simeq V^*$. 
As $\otimes^d = \Gamma^{1^d}$, from item~\ref{item:wtspace} in Section~\ref{subsec:yoneda} we have 
$\Sch(X) = \hom_{\RepG}(\Gamma^{1^d}, X) \simeq$ the $1^d$-weight space of $X(k^d)$.
Using these identifications we get 
\begin{eqnarray*}
\Sch(\Gamma^{d,V})&=&\hom_{\RepG}(\otimes^d,\Gamma^{d,V})\\
&\simeq&1^d\text{-weight space of the }GL(k^d)\text{-module } \, \Gamma^d(V^* \otimes k^d)\\
&\simeq&\Gamma^d(V_1 \oplus V_2 \oplus \cdots \oplus V_d)_{1^d}.
\end{eqnarray*}
By the exponential property $\Gamma(U \oplus W) \simeq \Gamma(U) \otimes \Gamma(W)$, we have
$$\Gamma^d(V_1 \oplus V_2 \oplus \cdots \oplus V_d) \simeq 
\bigoplus _{\lambda \in \Lambda(d,d)}\Gamma^{\lambda_1}(V_1) \otimes \cdots \otimes \Gamma^{\lambda_d}(V_d).$$
In this decomposition, 
$\Gamma^{\lambda_1}(V_1) \otimes \Gamma^{\lambda_2}(V_2) \otimes \cdots \otimes \Gamma^{\lambda_d}(V_d)$ is the 
$\lambda=(\lambda_1,\lambda_2,\ldots,\lambda_d)$-weight space of
the $GL(k^d)$-module $\Gamma^d(V^* \otimes k^d)$. 
As we require the $1^d=(1,1,\ldots,1)$-weight space, we have 
$\Sch(\Gamma^{d,V}) \simeq V_1\otimes \cdots \otimes V_d \simeq (V^*)^{\otimes d}$.

\medskip

One checks that the right action of $\eendo_{\RepG}(\otimes^d)$ by precomposition on $\hom_{\RepG}(\otimes^d,\Gamma^{d,V})$, 
after applying the isomorphisms with $kS_d^{op}$ and $(V^*)^{\otimes d}$ respectively, 
translates into the claimed left $kS_d$-action by permuting the tensor factors. 
\end{proof}

\begin{remark}
Similar arguments show that (i) $\Sch(\sym^{d,V})=(V^*)^{\otimes d}$ with the same $kS_d$-action and 
(ii) $\Sch(\wedge^{d,V})=(V^*)^{\otimes d}$ with the $kS_d$-action twisted by the sign, owing to the sign involved in the exponential property for the exterior algebra.
\end{remark}

\begin{theorem}
\label{thm:main}
For strict polynomial functors $X,Y$ of degree $d$, we have natural isomorphisms 
$$\Sch(X \intP Y)\simeq\Sch(X)\otimes \Sch(Y) \qquad {\rm and} \qquad
\Sch(\Bbb H (X,Y))\simeq\hom(\Sch(X), \Sch(Y)),$$
where the $\otimes$ and $\hom$ on the right hand side are calculated in $k$-{\rm Mod} and  
the action of $kS_d$ on these is defined in the usual way.
\end{theorem}

\begin{proof}

First consider the case when $X=\Gamma^{d,V}$ and $Y=\Gamma^{d,W}$ are representable functors.  
From Proposition~\ref{prop:main} we have 
$$\Sch(\Gamma^{d,V} \intP \Gamma^{d,W}) \simeq \Sch(\Gamma^{d,V \otimes W})
\simeq((V\otimes W)^{*})^{\otimes d}
\simeq (V^*)^{\otimes {d}}\otimes (W^*)^{\otimes {d}}
\simeq \Sch(\Gamma^{d,V})\otimes \Sch(\Gamma^{d,W}).
$$
Similarly we have $\Sch( \Bbb H (\Gamma^{d,V} , \Gamma^{d,W}) ) \simeq \hom (\Sch(\Gamma^{d,V}) , \Sch(\Gamma^{d,W})).$

\medskip

For general $X$ and $Y$ we express them as colimits of representable functors and use the fact that 
$\Sch$, $\intP$ and the Kronecker $\otimes$ product  preserve colimits, since they are left adjoints.
 \begin{eqnarray*}
 \Sch(X \intP Y) &\simeq& \Sch \Big(\nonrep{V}{X} \ \nonrep{W}{Y} \ \Gamma^{d,V \otimes W} \Big)\\
&\simeq & \nonrep{V}{X} \ \nonrep{W}{Y} \ \Sch \big(\Gamma^{d,V \otimes W} \big)\\
&\simeq & \nonrep{V}{X} \ \nonrep{W}{Y} \ \Sch \big(\Gamma^{d,V} \big) \otimes \Sch \big(\Gamma^{d,W} \big) \\
&\simeq &  \Sch \Big( \nonrep{V}{X} \ \Gamma^{d,V} \Big) \otimes \Sch \Big( \nonrep{W}{Y} \ \Gamma^{d,W} \Big) \\
&\simeq& \Sch(X) \otimes \Sch(Y).
 \end{eqnarray*}
The proof for $\Bbb H$ proceeds the same way. 
For the second isomporphism, one observes additionally that Sch preserves limits 
and for the fourth isomporphism that $\hom(\Sch (\Gamma^{d,V} ), - )$ preserves colimits 
because $\Sch (\Gamma^{d,V}) \simeq {V^*}^{\otimes d}$ is a finitely generated projective $k$-module.
\end{proof}

The preceding result continues to hold when all functors are replaced by their derived analogues.
This is essentially because $\Sch (\Gamma^{d,V}) \simeq {V^*}^{\otimes d}$ is a projective $k$-module 
and so (even though it is not a projective $kS_d$-module in general) it is acyclic for $\otimes$ and $\hom$.
This gives a proof via Grothendieck spectral sequence in appropriately bounded derived categories.
We will need a little extra work because, as in~\cite{Krause13}, we will work with unbounded derived categories.  
One gets the result for unbounded complexes using truncations and homotopy (co)limits, see~\cite{Spalt88, BN}.

\begin{theorem}
\label{thm:intderived} Let $X,Y$ be objects in the unbounded derived category ${\rm D} (\XRepG^{d}_{k})$. 
Then we have natural isomorphisms
 $$\Sch(X\overset{\bf L}\intP Y)\simeq \Sch(X)\overset{\bf L}\otimes \Sch(Y) \quad {\rm and} \quad 
\Sch( {\bf R}\Bbb H (X, Y))\simeq {\bf R}{\rm Hom}(\Sch(X), \Sch(Y)).$$
\end{theorem}

\begin{proof} We first recall some generalities, see~\cite{Spalt88, BN} for these matters. 
$\RepG$ and $kS_d$-Mod are Grothendieck categories. 
There are the usual quotient functors 
${\rm K} (\RepG) \overset {Q_{1}} \longrightarrow {\rm D} (\RepG)$ and  ${\rm K}(kS_d$-Mod$) \overset {Q_{2}} \longrightarrow {\rm K}(kS_d$-Mod). 
Each $Q_{j}$ has a left adjoint ${\bf p}_{j}$ (``K-projective resolution" -- used to calculate left derived functors) 
and a right adjoint ${\bf i}_{j}$ (``K-injective resolution" -- used to calculate right derived functors).
Arbitrary products and coproducts exist in ${\rm K} (\RepG)$ and ${\rm K}(kS_d$-Mod),
which pass to the respective derived categories by the quotient functors. 
Coproducts are preserved by ${\bf p}_{j}$ and products by ${\bf i}_{j}$.

\medskip

The functors $\Sch$, $\intP, \Bbb H, \otimes$, and $\hom$ pass to respective homotopy categories
(where one uses the same notation for them) and have derived functors. 
As $\Sch$ is exact, it extends to the derived category by termwise application, the extension still denoted by $\Sch$. 
It preserves (co)products.
Derived bifunctors of $\intP$ and $\otimes$ (respectively, $\Bbb H$ and $\hom$) are calculated by taking 
double complexes obtained from appropriate resolutions and forming the corresponding total complexes via coproduct (respectively, product).
$\intP$ and $\otimes$ preserve coproducts in each variable. 
$\Bbb H$ and $\hom$ preserve products in the second variable and turn coproducts in the first variable into products.
These properties pass to the respective derived functors because, e.g., ${\bf R} \Bbb H (X, -) \simeq Q_1 \circ \Bbb H (X, -) \circ {\bf i}_1$.
(To avoid clutter we will continue to denote $Q_1X$ in ${\rm D} (\RepG)$ by $X$.)

\medskip

First we prove the result about tensor products.
By Theorem~\ref{thm:main} we have the following isomorphism of bifunctors 
from $\RepG \times  \RepG$ to $kS_d$-Mod.
$$F:=\Sch \circ (- \intP-) \simeq G \circ (\Sch\times\Sch) \quad {\rm where} \quad G:=-\otimes- \ $$
As $\Sch$ is exact,  ${\bf L}F \simeq \Sch\circ(-\overset{\bf L}\intP -)$.
By the universal property of the derived bifunctor ${\bf L}F$, there is a natural transformation of triangulated bifunctors
$$\eta: {\bf L} G \circ (\Sch\times\Sch) = (-\overset{\bf L}\otimes -) \circ (\Sch\times\Sch) \to {\bf L}F$$ 
%
We will prove that $\eta$ is an isomorphism in two steps. 

\medskip

{\it Step 1.} We claim that the restriction of $\eta$ to ${\rm D}^- (\RepG)\times {\rm D}^- (\RepG)$ is an isomorphism. 
Consider the full subcategory $\mathcal{P}$ of $\RepG$ consisting of direct sums of representable objects $\Gamma^{d,V}$.
We observe below that $\Sch(\mathcal{P}) \times  \Sch(\RepG)$ is $(-\otimes -)$-projective in the terminology of~\cite[Definitions 10.3.9, 13.4.2]{KS}.
This would prove the claim by a suitable analogue of Grothendieck spectral sequence as formulated in~\cite[Proposition 13.3.13(ii)]{KS}.
(We need the analogue when, in the terminology there, $F$ is a product of two functors between abelian categories and $F'$ is a bifunctor.)

\medskip 

To see the $(-\otimes -)$-projectivity of $\Sch(\mathcal{P}) \times  \Sch(\RepG)$, 
first note that direct sums of objects of the form $\Sch(\Gamma^{d,V}) \simeq (V^*)^{\otimes d}$ form a generating subcategory of $kS_d$-Mod 
because, e.g., the projective generator $kS_d \simeq \Sch(\Gamma^{1^d}) \overset {\oplus} \hookrightarrow \Sch(\Gamma^{d,k^d})$. 
So one needs to check that for $X$ and $Y$ in ${\rm K^-}(kS_d$-Mod), the total complex $X \otimes Y$ is exact whenever one
of the following two conditions is met.
(i) all nonzero terms of $X$ are of the form $\oplus_{\alpha} (V_{\alpha}^*)^{\otimes d}$ and $Y$ is exact; or
(ii) $X$ is an exact complex all of whose nonzero terms are of the form $\oplus_{\alpha} (V_{\alpha}^*)^{\otimes d}$ and $Y$ is arbitrary.
As $(V_{\alpha}^*)^{\otimes d}$ is a projective $k$-module, 
this amounts to exactness of the total complex of a fourth quadrant double complex with exact columns/rows.

\medskip

{\it Step 2.} Fix an $X$ in ${\rm D}^-(\RepG)$. 
Let $\mathcal{C}$ be the full subcategory of ${\rm D}(\RepG)$ whose objects are those $Y$ for which $\eta(X,Y)$ is an isomorphism. 
Clearly $\mathcal{C}$ is a triangulated subcategory and by Step 1 it contains ${\rm D}^- (\RepG)$.
As ${\bf L} F$ and ${\bf L} G \circ (\Sch\times\Sch)$ preserve coproducts in each slot, $\mathcal{C}$ is closed under formation of arbitrary coproducts. 
It follows that $\mathcal{C}$ is all of ${\rm D}(\RepG)$
because any object $Y \simeq$ the cone of ($1 - \, $shift) endomorphism of the coproduct 
$\oplus_{n \geq 0}  \, Y^{\leq n}$ of bounded above truncations of $Y$~\cite{Spalt88},~\cite[Section 2]{BN}.
Repeating the argument by fixing an arbitrary $Y$ gives the result for arbitrary $X$ and $Y$.
 
\medskip
 
The result about $\Bbb H$ is obtained similarly. 
We indicate only the changes.
By Theorem~\ref{thm:main} we have the isomorphism of bifunctors 
from $(\RepG)^{op} \times  \RepG$ to $kS_d$-Mod
$$F':=\Sch \circ \, \Bbb H(- , -) \simeq G' \circ (\Sch\times\Sch) \quad {\rm where} \quad G':= \hom(-, -) \ .$$
leading to a natural transformation of triangulated bifunctors
$$\eta':  {\bf R}F' \to {\bf R} G' \circ (\Sch\times\Sch) = {\bf R}\hom (-, -) \circ (\Sch\times\Sch).$$ 
As before, the restriction of $\eta'$ to ${\rm D}^- (\RepG) \times {\rm D}^+ (\RepG)$ is an isomorphism. 
One sees this by using $\hom(- , -)$-injectivity of $\Sch(\mathcal{P}) \times  \Sch(\RepG)$, which follows by
exactness of the total complex of a first quadrant double complex with exact columns/rows.

\medskip

The triangulated bifunctors ${\bf R} F'$ and ${\bf R} G' \circ (\Sch\times\Sch)$ preserve products in the second slot and turn coproducts in the first slot into products.
By using homotopy limits in the second slot and homotopy colimits in the first slot, one gets as before that $\eta'$ is an isomorphism on 
all of ${\rm D} ((\RepG)^{op}) \times {\rm D} (\RepG)$.
\end{proof}

We note an interesting consequence of Theorem~\ref{thm:intderived} for modular representations. 
If $k$ is a field then the Kronecker product is exact. 
Therefore, for $X$ and $Y$ in $\RepG$,
$$\Sch(\higheri{i}(X\overset{\bf L}\intP Y)) \simeq 
\higheri{i} (\Sch(X\overset{\bf L}\intP Y)) \simeq 
\higheri{i}(\Sch(X)\overset{\bf L}\otimes \Sch(Y))=0 \ \text {\rm for all} \ i \neq 0.$$ 
This is uninteresting when the 
char $k = 0$, as in that case $\RepG$ is semisimple and therefore $\intP$ is also exact. 
However, if char $k = p > 0$, then $\RepG$ is not semisimple. 
Its simple objects $L_k(\lambda)$ are still indexed by all partitions $\lambda$ of $d$. 
For $p$-restricted partitions $\lambda$, the $kS_d$-modules $\Sch(L_k(\lambda))$ form a complete set of simple objects and 
for other $\lambda$, one has $\Sch(L_k(\lambda)) =0$. 
Thus we get
\begin{corollary}\label{cor:restricted}
If $k$ is a field of characteristic $p$ then for finitely generated $X,Y\in \XRepG^{d}_{k},$ all composition factors of 
$\higheri{i}(X\overset{\bf L}\intP Y)$ have non $p$-restricted highest weights for $i \neq 0$.
\end{corollary}

A compatibility of the internal tensor product and internal hom with adjoints of the Schur functor was discovered recently by Reischuk.
We prove below somewhat stronger versions of her results (compare~\cite[Corollaries 4.4 and 5.5]{Rei16}) by different means using Theorem~\ref{thm:main} and standard adjunctions.
Moreover, being purely formal, our method extends to give new results in the derived setting by using Theorem~\ref{thm:intderived}. 

\begin{corollary}\label{cor:adjointinternal}
Let $M,N\in kS_{d}$-{\rm Mod}. Let $X$ be in $\XRepG^{d}_{k}$ such that $\Sch(X) \simeq M$. 
Then
$$\mathcal{L}(M\otimes N)\simeq X \intP \mathcal{L}N  \qquad {\rm and} \qquad
\mathcal{R} (\hom(M,N)) \simeq \mathbb{H}(X,\mathcal{R}N).$$ 
The isomorphisms are functorial in $N$.  
They are functorial in $M$ as well if $X = {\mathcal L} M$ or if $X = {\mathcal R} M$.

\end{corollary}

\begin{proof}
The second result follows from the Yoneda lemma by the functorial isomorphisms
\begin{align*}
\hom_{\RepG}(Z,\mathcal{R}(\text{Hom}(M,N)))  
& \simeq  \text{Hom}_{kS_{d}}(\text{Sch}(Z),\text{Hom}(M,N)) && \text {by }  \Sch \vdash \mathcal {R} \\
& \simeq  \text{Hom}_{kS_{d}}(\text{Sch}(Z)\otimes M,N) &&  \text {by }  \otimes \vdash \hom \\ 
& \simeq  \text{Hom}_{kS_{d}}(\text{Sch}(Z)\otimes\text{Sch}(X),N) && \text {by choice of } X \\
& \simeq  \text{Hom}_{kS_{d}}(\text{Sch}(Z\intP X),N) &&  \text {by Theorem } \ref{thm:main} \\
& \simeq  \text{Hom}_{\RepG}(Z\intP X, \mathcal{R}N)  &&  \text {by }  \Sch \vdash \mathcal {R} \\
& \simeq  \text{Hom}_{\RepG}(Z,\mathbb{H}(X,\mathcal{R}N)) &&  \text {by } \intP \vdash \Bbb H.
\end{align*}
The proof of the first result is entirely parallel by using a test object in the second slot, see the next proof. 
The last sentence follows from the isomorphisms $\Sch \circ \, \mathcal L \simeq id \simeq \Sch \circ \, \mathcal R$.
\end{proof}

We show that the same result is true at the derived level by using Theorem~\ref{thm:intderived}. 
To simplify notation, we will denote the left derived functor of $\mathcal L$ by $\mathcal L'$ instead of ${\bf L} \mathcal L$
and likewise the right derived functor of $\mathcal R$ by $\mathcal R'$ instead of ${\bf R} \mathcal R$.

\begin{corollary}\label{cor:derivedadjointinternal}
Let $M,N \in {\rm D(}kS_{d}$-{\rm Mod)}. Let $X$ be in ${\rm D}(\XRepG^{d}_{k})$ such that $\Sch(X) \simeq M$. Then
 $$\mathcal{L}' (M \overset{{\bf L}}\otimes N)\simeq X \overset{{\bf L}}\intP \mathcal{L}'N \qquad {\rm and} \qquad
 \mathcal{R}' ( {\bf R}{\rm Hom}(M,N) ) \simeq {\bf R} \mathbb{H}(X,\mathcal{R}'N).$$ 
The isomorphisms are functorial in $N$  
and functorial in $M$ as well if $X = {\mathcal L}' M$ or if $X = {\mathcal R}' M$.
\end{corollary}

 \begin{proof} 
 
In general, an adjunction $F \vdash G$ of additive functors between Grothedieck categories leads 
to an adjunction ${\bf L} F \vdash {\bf R} G$ between the corresponding derived categories. 
(The adjunction first passes to homotopy categories and then one combines with the adjunctions ${\bf p} \vdash Q \vdash {\bf i}$, 
e.g. see~\cite[Proposition 4.1]{Krause13} for the case of $(X \intP -)  \vdash \Bbb H (X, -)$.)
Thus we may repeat the argument in the proof of Corollary~\ref{cor:adjointinternal} by using derived versions of all adjunctions and Theorem~\ref{thm:intderived}.
We sketch this for the first isomorphism.
\begin{align*}
\hom_{{\rm D}(\RepG)}(\mathcal{L}' (M \overset{{\bf L}}\otimes N), Z)  
& \simeq  \text{Hom}_{{\rm D}(kS_{d}\text {-Mod})}(N, {\bf R}\text{Hom}(M, \text{Sch}(Z)) 
    &&  \text {by }   \mathcal {L}' \vdash \Sch, \  \overset{\bf L}  \otimes \vdash {\bf R}\text{Hom} \\ 
& \simeq  \text{Hom}_{{\rm D}(kS_{d}\text {-Mod})}(N, \text{Sch}({\bf R}\Bbb{H}(X, Z)) 
    && \text {by Theorem } \ref{thm:intderived}\\  
& \simeq  \text{Hom}_{{\rm D}({\RepG})}(X \overset{{\bf L}}\intP \mathcal{L}'N, Z) 
    &&  \text {by }   \mathcal {L}' \vdash \Sch, \  \overset{\bf L}  \intP \vdash {\bf R}\Bbb{H}.
\end{align*}
Finally, as $\Sch$ is exact, deriving $\Sch \circ \, \mathcal L \simeq id \simeq \Sch \circ \, \mathcal R$ gives $\Sch \circ \, \mathcal L' \simeq id \simeq \Sch \circ \, \mathcal R '$. This gives the last assertion. 
\end{proof}


\subsection{Internal tensor product and weight spaces}\label{sec:paramet}

In this section we compute, for an arbitrary $X$ in $\RepG$, the internal tensor product $X \intP \Gamma^{\lambda}$. 
Taking $X = \Gamma^{\mu}$, we answer a question posed by Krause 
by finding an explicit expression for $\Gamma^{\mu} \intP \Gamma^{\lambda}$. 
Similar calculations of $\wedge^\lambda \intP \wedge^\mu$, $\sym^\lambda \intP \sym^\mu$ and $\wedge^\lambda \intP \sym^\mu$ follow.

\medskip

Finding of $X \intP \Gamma^{\lambda}$ is an important step to calculate $X \intP Y$ for other $Y$:  one often tries to 
resolve $Y$ by direct sums of functors of the type  $\Gamma^{\lambda}$ and then one can use right exactness of $\intP$.
 
\medskip
 
Note that for $X$ in $\RepG$, $X_V(W) := X(V \otimes W)$ is functorial in each of the two variables $V$ and $W$. 
In particular $X(V \otimes W)$ has a $GL(V) \times GL(W)$-action. 
When $V=k^n$, any weight space of $X(k^n \otimes W)$ with respect to the action of $GL(k^n)$ is still functorial in $W$ and 
thus yields a strict polynomial functor of degree $d$. 
It will be convenient to formalize this as follows. 
(Compare the equivalent definition of $F^\lambda$ in~\cite[Section 5]{Chalupnik05}.)
\begin{definition}\label{defn:param} 
Let $X$ be a strict polynomial functor of degree $d$. Let $\lambda \in \Lambda(n, d)$. 
Then we define a strict polynomial functor $X^{\lambda}$ by 
$X^{\lambda}(V) = X(k^{n}\otimes V)_\lambda$, i.e., the $\lambda$-weight space of the $GL(k^{n})$-module $X(k^{n}\otimes V).$ 
\end{definition}
Observe that $(\Gamma^d)^\lambda$ is precisely $\Gamma^\lambda$ and likewise for symmetric and exterior powers. 
We have the following basic calculation, generalizing~\cite[Proposition 3.4, Corollary 3.7]{Krause13}.
\begin{lemma}\label{lm:paramet}
For $\lambda\in\Lambda(n,d)$, we have $X \intP \Gamma^\lambda \simeq  X^{\lambda}.$
\end{lemma}
\begin{proof}
First let $X=\Gamma^{d,V}$. 
We have $(\Gamma^{d,V})^{\lambda}(U) =$ the $\lambda$-weight space of the $GL(k^n)$-module
$\Gamma^{d,V}(k^n \otimes U ) =  \Gamma^{d}(V^{*} \otimes k^n \otimes  U)_\lambda$. 
Using~\cite[Lemma 2.5]{Krause13}, 
$(\Gamma^{d,V} \intP \Gamma^\lambda) (U) \simeq \Gamma^\lambda(\Gamma^{d,V}(U)) 
= \Gamma^\lambda( V^{*} \otimes U) = \Gamma^{d}(k^n \otimes V^{*} \otimes  U)_\lambda$, 
where the $\lambda$-weight space is again taken for the $GL(k^n)$-action.
In general we again use $X\simeq \nonrep{V}{X} \Gamma^{d,V}$ and check compatibility with colimits. 
We have
$$
\bigg( \nonrep{V}{X} \Gamma^{d,V} \bigg) \intP \; \Gamma^{\lambda}  
\quad \simeq \quad \nonrep{V}{X}(\Gamma^{d,V} \intP \Gamma^{\lambda})
\quad \simeq \quad \nonrep{V}{X}(\Gamma^{d,V})^{\lambda}
\quad \simeq \quad \bigg(\nonrep{V}{X} \Gamma^{d,V}\bigg)^{\lambda}.
$$
The first isomorphism is true since $\intP$ commutes with colimits. 
The second follows from the special case proved above. 
Finally, the functor $Y \rightarrow Y^{\lambda}$ preserves colimits, since it is exact 
and preserves arbitrary direct sums in the abelian category $\RepG$. 
\end{proof}

Let $\mu \in \Lambda(n,d)$ and let $\lambda \in \Lambda(m,d)$. 
Let $S$ denote the set of $n$ by $m$ matrices with non-negative integer entries, with row sums $\mu$ and column sums $\lambda$. 
Every such matrix $S$ naturally gives us a $\nu \in \Lambda(mn,d)$. 
\begin{proposition} \label{prop:lambdamu}
Let $\mu \in \Lambda(n,d)$ and let $\lambda \in \Lambda(m,d)$. Then,
\begin{center}
$\Gamma^{\mu} \intP \Gamma^{\lambda} \simeq \underset{\nu \in S}\bigoplus\Gamma^{\nu}$.
\end{center}
\end{proposition}

\begin{proof}
We have $\Gamma^{\mu} \intP \Gamma^{\lambda} \simeq (\Gamma^{\mu})^{\lambda}$. 
To unwind this, let $U\in \Gamma^{d}P_{k}$. 
Then using definition~\ref{defn:param} we get
$$
(\Gamma^{\mu} \intP \Gamma^{\lambda})(U)
\quad \simeq \quad  \Gamma^{\mu}({k^{m}}\otimes U)_{\lambda}
\quad \simeq \quad  (\Gamma^{d}({k^{n}}\otimes {k^{m}}\otimes U)_{\mu})_{\lambda},
$$
where the $\lambda$-weight space is taken with respect to $GL({k^{m}})$ and 
the $\mu$-weight space is taken with respect to $GL({k^{n}})$.
Thus, $(\Gamma^{\mu}\intP\Gamma^{\lambda})(U)$ is the $(\mu, \lambda)$-weight space of 
$\Gamma^{d}(k^{n}\otimes k^{m}\otimes U)$ with respect to the action of $GL({k^{n}})\times GL({k^{m}}).$

\medskip

On the other hand $\Gamma^d(k^{n}\otimes k^{m} \otimes U)$ is a polynomial representation of $GL({k^{n}}\otimes {k^{m}})$ 
whose weights are given by $n\times m$ non-negative integer matrices $(\nu_{ij})$ with the entries adding up to $d$.
We can pull back the action of $GL({k^{n}}\otimes {k^{m}})$ to $GL({k^{n}})\times GL({k^{m}})$ 
via the morphism $A \times B \mapsto A \otimes B$. 
This induces another weight space decompositiom of $\Gamma^d(k^{n}\otimes k^{m} \otimes U)$, 
in which elements of the $(\nu_{ij})$-weight space under the $GL({k^{n}}\otimes {k^{m}})$-action now have weight 
$(\mu_{1},\ldots,\mu_{n},\lambda_{1},\ldots,\lambda_{m})$ for the $GL({k^{n}})\times GL({k^{m}})$-action, 
where $\mu_{i}= \sum_{j=1}^{m}\nu_{ij},$ and $\beta_{j}=$ $\sum_{i=1}^{n}\nu_{ij}.$
\end{proof}

This leads to the calculation of all internal tensor products of the form $X^\lambda \intP Y^\mu$, 
where $X$ and $Y$ are classical exponential functors $\Gamma^d, \wedge^d$ or $\sym^d$.  
For this we need the three basic calculations $\wedge^d \intP \wedge^d$, $\wedge^d \intP \sym^d$ and $\sym^d \intP \sym^d$. 
(We already know that $\Gamma^d \intP X \simeq X$ for arbitrary $X$.) 
We first separate out a useful remark.

\begin{remark} 
Calculating the effect of $X \intP -$ on a morphism between representable functors.
A morphism $\sigma: \Gamma^{d,V} \to \Gamma^{d,W}$ corresponds via the Yoneda embedding 
to $\phi \in \hom_{\Gamma^d P_k}(W,V) = \Gamma^d \hom(W,V)$. 
We have 
$X \intP \Gamma^{d,V} \simeq X^V$ and $X \intP \Gamma^{d,W} \simeq X^W$.
Then the map $(X \intP \, \sigma) (U): X^V(U) \to X^W(U)$ is the map $X(\phi^* \otimes id_U^{\otimes d})$. 
One may first check this for representable $X$ and then deduce it for general $X$ by compatibility of $\intP$ with colimits.
\end{remark}

The first calculation in the following result is ~\cite[Proposition 3.6]{Krause13}. 
For the rest we follow the method used for parallel calculations of $\Bbb H$ by Touz\'e in~\cite[Lemma 4.6]{Touze14} along with the preceding remark. 

\begin{proposition}\label{prop:wedgesym}
We have $\wedge^{d}\intP\wedge^{d}\simeq\sym^{d}$ and $\sym^{d}\intP\sym^{d}\simeq\sym^{d}$. 
If 2 is invertible in $k$ then $\sym^{d}\intP\wedge^{d} \simeq \wedge^{d}$. 
If $2 = 0$ in $k$ then $\sym^{d}\intP\wedge^{d} \simeq \sym^{d}$. 
\end{proposition}

\begin{proof} 
(We can use the formula $\mathcal L(\Sch X) \simeq X \intP \sym^d$ from Remark~\ref{remark:adjointinternal} 
with $X = \sym^d$ and $X = \wedge^d$, but we prefer to give a direct argument.) 
In $\RepG$ we have the presentation
$$ \underset {i = 0} {\overset {d-1} \oplus} \ \otimes^d  \;  \xrightarrow {\oplus (1 - \sigma_i)} \; \otimes^d \longrightarrow \sym^d \longrightarrow 0,$$
where $\sigma_i$ switches the $i$ and $i+1$ tensor factors. 
To apply $\wedge^d \ \intP \ - $ we realize $\otimes^d$ as a summand of $\Gamma^{d,k^d}$ via the exponential property. 
Then $\sigma_i$ can be realized as the restriction of the morphism $\tau_i$ in $\eendo_{\RepG}(\Gamma^{d,k^d})$ that 
corresponds to $f_i^{\otimes d} \in \Gamma^d \eendo (k^d)$, 
where $f_i \in \eendo (k^d)$ switches the standard basis vectors $e_i$ and $e_{i+1}$ in $k^d$.
Now $\wedge^d \ \intP \ \otimes^d \simeq \otimes^d$ by Lemma~\ref{lm:paramet} and the exponential property of $\wedge^d$.
The map $\wedge^d \ \intP \ \tau_i (U) \in \eendo \wedge^{d,k^d}(U)$ switches the occurrences of $e^*_i$ and $e^*_{i+1}$
in $d$-fold wedge products of vectors $e^*_j \otimes u \in (k^{d})^* \otimes U$. 
Identifying $U^{\otimes d}$ inside $\wedge^d((k^d)^* \otimes U)$ gives $\wedge^d \ \intP \ \sigma_i = - \sigma_i$.
The claims about $\wedge^d \intP \sym^d$ follow from the resulting presentation
$$ \underset {i = 0} {\overset {d-1} \oplus} \ \otimes^d  \;  \xrightarrow {\oplus (1 + \sigma_i)} \; 
\otimes^d \longrightarrow \wedge^d \intP \sym^d \longrightarrow 0.$$
Applying $\wedge^d \ \intP \ - $ again gives a presentation of $\sym^d \intP \sym^d$, as $\wedge^{d}\intP\wedge^{d}\simeq\sym^{d}$. 
But this is also the original presentation as the sign of $\sigma_i$ switches once again.  
\end{proof}

\begin{remark}\label{2problem} 
If 2 is a nonzero nonunit, then $\wedge^d \intP \sym^d$ can be more complicated, as can be seen when $k= \Bbb Z$ and $d=2$. 
In this case we have the exact sequence $0 \to I^{(1)} \to \wedge^2 \intP \sym^2 \to \wedge^2 \to 0,$ 
where $I^{(1)}$ is the 2-torsion functor defined by 
$I^{(1)}(V) = \Bbb Z$-span of $\{ v \otimes v \} / \Bbb Z$-span of $\{ v_1 \otimes v_2 +  v_2 \otimes v_1 \} $. 
The functor $I^{(1)}$ can be identified with the Frobenius twist of the identity functor $I$ over $k = \Bbb Z / 2 \Bbb Z$. 
In contrast, over any hereditary $k$, if functors $X$ and $Y$ take values in $P_k$, then so does $\Bbb H(X,Y)$. 
\end{remark}

From Propositions~\ref{prop:lambdamu} and~\ref{prop:wedgesym} we have the following calculations, 
where all direct sums are over the same set $S$ as in Proposition~\ref{prop:lambdamu}.

\begin{corollary}\label{cor:symandwedge}

Let $\lambda\in\Lambda(m,d)$ and $\mu\in\Lambda(n,d)$. Then we have 
$$\Gamma^{\lambda}\intP\wedge^{\mu}\simeq\displaystyle\oplus\wedge^{\nu}  \quad , \quad
\Gamma^{\lambda}\intP \sym^{\mu}\simeq \wedge^{\lambda}\intP\wedge^{\mu}\simeq  
\sym^{\lambda}\intP\sym^{\mu}\simeq \displaystyle\oplus \sym^{\nu},$$
$$\sym^{\lambda}\intP\wedge^{\mu}\simeq \displaystyle\oplus\wedge^{\nu} \text{ if 2 is a unit and }  
\sym^{\lambda}\intP\wedge^{\mu} \simeq \displaystyle\oplus\sym^{\nu} \text{ if } 2=0.$$

\end{corollary}

Applying $\Sch$ to the above along with Theorem~\ref{thm:main} gives analogous descriptions of Kronecker products 
involving permutation modules $M^{\lambda}$ and signed permutation modules. 
These are well known when $k$ is a field  of characteristic 0 in the context of symmetric functions~\cite[Chapter I, 7.23e]{MacD}.
For example, we have

\begin{corollary}
Let $\mu \in \Lambda(n,d)$ and let $\lambda \in \Lambda(m,d)$. 
Then, $M^{\mu}\otimes M^{\lambda} \simeq \oplus_{\nu \in S} M^{\nu}$ in $kS_d$\rm{-Mod}. 
\end{corollary}


\subsection{Internal tensor product and Weyl functors}\label{sec:intweyl}

We show that $\Delta(\lambda) \intP \Gamma^\nu$ has an explicit Weyl filtration. 
We obtain parallel results for dual Weyl functors using Koszul duality and for (dual) Specht modules using the Schur functor.
By contrast, the internal tensor product of two Weyl functors need not have a Weyl filtration 
e.g., $\wedge^d \intP \wedge^d \simeq \sym^d$ by Proposition~\ref{prop:wedgesym}. 
But we show that their higher derived internal tensor products do vanish, which is not true for two dual Weyl functors.

\begin{proposition}\label{prop:weyl}
$\Delta(\lambda)\intP \Gamma^{\nu}$ has a Weyl filtration that is independent of the ground ring $k$ and in which 
the multiplicity of any Weyl functor can be calculated as a sum of products of Littlewood-Richardson coefficients.
\end{proposition}

\begin{proof}
More generally, we will give a precise description of $\Delta(\lambda/\mu)\intP\Gamma^{\nu},$ 
where $\Delta(\lambda/\mu)$ is the skew Weyl functor~\cite{ABW} corresponding to
partitions $\lambda$ and $\mu$ with $\mu\subset \lambda$, i.e., the Young diagram of $\mu$ is contained in that of $\lambda$.
Let $\nu = (\nu_1, \ldots, \nu_n) \in \Lambda(n,d)$. By Lemma~\ref{lm:paramet}, 
\begin{equation}\label{weyl1}
(\Delta(\lambda/\mu)\intP\Gamma^{\nu})(V) \simeq
\Delta(\lambda/\mu)^{\nu}(V) \simeq \Delta(\lambda/\mu)(\hom
(k^n, V))_{\nu} \simeq \Delta(\lambda/\mu)(V\oplus V\oplus \cdots \oplus
V)_{\nu}. \end{equation}

Let $U$ and $W$ be free $k$-modules of finite rank. 
In~\cite{AB85}, Akin and Buchsbaum give an explicit construction of a filtration of the skew Weyl module 
$\Delta(\lambda/\mu)(U \oplus W)$ that is universal (i.e., independent of the ground ring $k$), 
functorial in $U$ and $W$ and whose associated graded object is  
\begin{equation}\label{weyl2}
\bigoplus_{\underset{\mu \subset \alpha \subset \lambda}
{\text {partitions } \alpha}} \Delta(\alpha/\mu)(U) \otimes \Delta(\lambda/\alpha)(W).
\end{equation} 
We use~(\ref{weyl2}) to first calculate  $\Delta(\lambda/\mu) \intP \Gamma^{d,k^n} \simeq \Delta(\lambda/\mu)^{k^{n}}$. 
Taking $V$ to be free of finite rank and using~(\ref{weyl2}) repeatedly 
gives a filtration of $\Delta(\lambda/\mu) \intP \Gamma^{d,k^n} (V) \simeq \Delta(\lambda/\mu)(V\oplus V\oplus \cdots \oplus V)$. 
As this description is true for any free module $V$ of finite rank, 
the same description is valid in $\RepG$ by the equivalence~\ref{thm:FS}. 
Altogether, we get a filtration of $\Delta(\lambda/\mu) \intP \Gamma^{d,k^n}$ whose associated graded object is
\begin{equation}\label{eqn:weyl3} 
\bigoplus_{\underset{\mu \subset \alpha^1 \subset \alpha^2 \subset \cdots \subset \alpha^{n-1} \subset
\lambda}{\text {partitions } \alpha^1, \alpha^2, \ldots, \alpha^{n-1}}} 
\Delta(\alpha^1/\mu) \otimes \Delta(\alpha^2/\alpha^1) \otimes \cdots \otimes\Delta(\lambda/\alpha^{n-1}). 
\end{equation}
Taking the $\nu$-weight space in~(\ref{weyl1}) is equivalent to requiring $|\alpha^i/\alpha^{i-1}| = \nu_i$ for all $i = 1, \ldots, n$
in~(\ref{eqn:weyl3}), with the understanding that $\alpha^0 = \mu$ and $\alpha^n = \lambda$.
Therefore we get a filtration of $\Delta(\lambda/\mu)\intP\Gamma^{\nu}$ whose associated graded object is
\begin{equation}\label{weyl4}
\bigoplus_{\underset{\alpha^0 = \mu, \alpha^n = \lambda, |\alpha^i/\alpha^{i-1}| = \nu_i}{\text {partitions } 
\alpha^0 \subset \alpha^1 \subset \alpha^2 \subset \cdots \subset \alpha^{n-1} \subset \alpha^n}} 
\Delta(\alpha^1/\mu) \otimes \Delta(\alpha^2/\alpha^1) \otimes \cdots \otimes\Delta(\lambda/\alpha^{n-1}).
\end{equation}
Finally note that each tensor product in~(\ref{weyl4}) itself has a Weyl filtration in which the multiplicity of 
any $\Delta(\beta)$ can be calculated as a sum of products of Littlewood-Richardson coefficients. 
This is because $\Delta(\alpha^i/\alpha^{i-1})$ has a Weyl filtration in which the multiplicity of $\Delta(\beta)$ 
equals the Littlewood-Richardson coefficient $c^{\alpha^i}_{\alpha^{i-1}\beta}$, 
by~\cite[Theorem 2.6]{Kouwenhoven} or contravariant dual of~\cite[Theorem 1.3]{Boffi}. 
Applying this in all tensor slots leads to a filtration whose successive quotients are $n$-fold tensor products of Weyl functors. 
These in turn have Weyl filtrations with multiplicities given by products of Littlewood-Richardson coefficients. 
\end{proof}

\begin{corollary}\label{cor:koz}
$\Delta(\lambda)\intP \wedge^{\nu} \simeq \nabla(\lambda^\prime)\intP \Gamma^{\nu}$ has a dual Weyl filtration 
that is independent of the ground ring $k$ and in which the multiplicity of any dual Weyl functor can be calculated 
as a sum of products of Littlewood-Richardson coefficients.
\end{corollary}

\begin{proof}
Apply $-\overset {\bf L} {\intP} \wedge^d$ to Proposition~\ref{prop:weyl}. 
By~\cite[Propositions 3.4, 4.16]{Krause13} we have, respectively, 
$\wedge^d \overset {\bf L}{\intP} \Gamma^\nu \simeq \wedge^\nu$ (or see Lemma~\ref{lm:paramet}) and
$\wedge^d \overset {\bf L}{\intP} \Delta(\lambda) \simeq \nabla(\lambda^\prime)$.
We calculate
\begin{center}
$\Delta(\lambda) \overset {\bf L}{\intP} \wedge^{\nu} \simeq  
\Delta(\lambda) \overset {\bf L}{\intP} \wedge^d \overset {\bf L}{\intP} \Gamma^{\nu} \simeq \nabla(\lambda^\prime) 
\overset {\bf L}{\intP} \Gamma^{\nu}$ 
\end{center}
and remark that applying $\wedge^d \intP -$ turns a Weyl filtration into a dual Weyl filtration, 
because a functor with a Weyl filtration is acyclic for the left exact functor $\wedge^d \intP -$.
\end{proof}

The property of having a (dual) Weyl filtration passes to direct summands. 
As any projective object in the $\RepG$ is a summand of a direct sum of $\Gamma^{d,k^n}$, we get from~\ref{prop:weyl} and~\ref{cor:koz} 

\begin{corollary}
The internal tensor product of a (dual) Weyl functor and a finitely generated projective object in $\XRepG^{d}_{k}$ has a (dual) Weyl filtration.
\end{corollary}

Applying the Schur functor to Proposition~\ref{prop:weyl} and Corollary~\ref{cor:koz} we obtain 
\begin{corollary}\label{cor:specht}
The Kronecker product of a (dual) Specht module with a permutation module has an explicitly constructed 
(dual) Specht filtration that is independent of the ground ring $k$. 
\end{corollary}

All prior results in this section involving $\intP$ have one of the arguments a projective object 
and so stay valid after replacing $\intP$ with $\overset {\bf L}{\intP}$. 
Likewise, $\wedge^d \overset {\bf L}{\intP} \Delta(\lambda) \simeq \nabla(\lambda^\prime)$~\cite[Proposition 4.16]{Krause13} 
implies in particular that higher derived internal tensor products of $\wedge^{d}$ and a Weyl functor vanish.
We show more generally that, even though $\Delta(\lambda) \intP \Delta(\mu)$ is hard to compute, 
it is in rep$\Gamma^d_k$ and the corresponding higher derived internal tensor products always vanish. 

\begin{proposition}
\label{prop:twoweyl}

Let $\Delta(\lambda)$ and $\Delta(\mu)$ be Weyl functors corresponding to partitions $\lambda$ and $\mu$ of $d$. 
Then $\higheri{i}(\Delta(\lambda)\overset{\bf L}\intP\Delta(\mu))=0$ for $ i \neq 0$ and
$\Delta(\lambda) \intP \Delta(\mu)$ is in {\rm rep}$\Gamma^d_k$.
\end{proposition}

\begin{proof}
We will use the derived duality $(-)^\diamond$ from~\cite[Section 4]{Krause13}, 
recall the discussion of duality in Section~\ref{subsec:duality}. 
We have $\Delta(\mu) \simeq \nabla(\mu)^\diamond$. 
(In the context of derived functors, single objects will be considered as complexes concentrated in degree 0.)
As $\Delta(\lambda)$ has a finite projective resolution by finite direct sums of various $\Gamma^\nu$ by~\cite[Section 4]{AB88}, 
we may use~\cite[Lemma 4.5]{Krause13} (see Lemma~\ref{lemma:duality}) to get 
$$\Delta(\lambda)  \overset{\bf L}\intP \Delta(\mu)  \simeq  {\bf R}\mathbb{H}(\Delta(\lambda),\nabla(\mu))^{\diamond}.$$
It is enough to prove two claims. 
\begin{enumerate} 
\item $\higheri{i}({\bf R}\mathbb{H}(\Delta(\lambda),\nabla(\mu))) = 0$ for $i \neq 0$, which would give
${\bf R}\mathbb{H}(\Delta(\lambda),\nabla(\mu))\simeq \mathbb{H}(\Delta(\lambda),\nabla(\mu))$. 
\item $\mathbb{H}(\Delta(\lambda),\nabla(\mu))$ is in rep$\Gamma^d_k$,  which would give 
$\mathbb{H}(\Delta(\lambda),\nabla(\mu))^{\diamond} \simeq \mathbb{H}(\Delta(\lambda),\nabla(\mu))^{\circ}$. 
\end{enumerate}
To prove the first claim we use the equivalence~\ref{thm:FS} and calculate using the isomorphism~(\ref{rh})
$$\higheri{i}({\bf R}\mathbb{H}(\Delta(\lambda),\nabla(\mu))(k^{d}))\simeq  \text{Ext}^{i}_{\RepG}(\Delta(\lambda)^{k^{d}},\nabla(\mu)).$$
These Ext$^i$ vanish for $i \neq 0$ by Proposition~\ref{prop:vanishing} because, 
by the discussion in Proposition~\ref{prop:weyl}, $\Delta(\lambda)^{k^{d}}$ has a Weyl filtration. 
For the second claim, by item 5 in Section~\ref{subsec:yoneda} it is enough to prove that $\mathbb{H}(\Delta(\lambda),\nabla(\mu))(k^{d})$ is in $P_{k}$.
We have 
\begin{center}
$\mathbb{H}(\Delta(\lambda),\nabla(\mu))(k^{d})\simeq \text{Hom}_{\RepG}(\Delta(\lambda)^{k^{d}},\nabla(\mu))
\simeq \text{Hom}_{S_{k}(d,d)}(\Delta(\lambda)^{k^{d}}(k^{d}),\nabla(\mu)(k^{d})),$
\end{center}
which is seen to be free of finite rank over $k$ as follows. When $k=\mathbb{Z}$ it is a subgroup of 
$\text{Hom}_\Bbb Z(\Delta(\lambda)^{\Bbb Z^{d}}(\Bbb Z^{d}),\nabla(\mu)({\Bbb Z^{d}}))$, which is a free abelian group of finite rank.
Now change base to $k$ \cite[Theorem 5.3]{AB88} and use vanishing of Ext$^1$ resulting from Proposition~\ref{prop:vanishing}.
The two claims together prove the result. 
 \end{proof}

\begin{remark}
Proposition~\ref{prop:twoweyl} is not true for dual Weyl functors. 
For example, let $k$ be of characteristic 2. Applying  $(- \, \intP \sym^{2})$ to the projective resolution
$0\rightarrow \Gamma^{2}\rightarrow \otimes^{2}\rightarrow \wedge^{2}\rightarrow 0$
of $\wedge^2$, we get from Proposition~\ref{prop:wedgesym} the sequence $0\rightarrow \sym^{2}\rightarrow \otimes^{2}\rightarrow \sym^{2} \rightarrow 0$.
It follows that $\higheri{-1}(\wedge^{2}\, \overset{\bf L} \intP\sym^{2})  \simeq I^{(1)},$ 
the Frobenius twist of the identity functor.
\end{remark} 

We record yet another corollary, one that requires 2 to be a unit in $k$. 
In that case we have $\wedge^d \intP \sym^d \simeq \wedge^d$ by Proposition~\ref{prop:wedgesym} and therefore 
$$\nabla(\lambda)\intP \sym^{\nu} \simeq (\Delta(\lambda')\intP \wedge^d ) \intP ( \sym^d \intP \Gamma^{\nu} )  \simeq 
\Delta(\lambda') \intP \wedge^d  \intP \Gamma^{\nu} \simeq \nabla(\lambda)\intP \Gamma^{\nu}.$$
As $\nabla(\lambda)\intP \Gamma^{\nu}$ has a dual Weyl filtration by Corollary~\ref{cor:koz}, we get

\begin{corollary}\label{cor:nablasym}
If 2 is a unit in $k$, then $\nabla(\lambda)\intP \sym^{\nu}$ has a dual Weyl filtration in which 
the multiplicity of any dual Weyl functor can be calculated as a sum of products of Littlewood-Richardson coefficients.
\end{corollary}


\subsection{Kronecker multiplicities for $S_d$  via the internal tensor product} 
\label{subsec:kron}

For this section, let $k$ be a field of characteristic 0. 
Now $\Sch$ is an equivalence of semisimple monoidal categories. 
The Kronecker problem for the symmetric group asks for a good description of 
multiplicities of Specht modules in the tensor product of two Specht modules. 
Via $\Sch$, this is equivalent to decomposing $\Delta(\lambda) \intP \Delta(\mu)$, 
where $\lambda$ and $\mu$ are partitions of $d$.
One can do this, e.g., by combining the Jacobi-Trudi formula to express $\Delta(\mu)$ as an alternating sum of 
various $\Gamma^\nu$ and then using ~\ref{prop:weyl} to calculate each $\Delta(\lambda) \intP \Gamma^\nu$.
Such an algorithm involves cancellations and its ingredients translate into standard facts about the internal product of symmetric functions. 
Even so, in two special cases (namely when $\mu$ is either a two-row partition or a hook) 
we show below that one can devise a reasonably simple procedure in terms of $\intP$.
In case of a hook, this uses a signed version of polynomial functors defined by Axtell~\cite{Axt13}.

\begin{example} 
We calculate $\Delta (\lambda)\intP\Delta((a,b))$, 
where $a\geq b$ are positive integers with $a+b = d$. 
\begin{eqnarray*}
\Delta(\lambda) \intP \Gamma^{(a,b)}  & \simeq & 
\Delta(\lambda)_{(a,b)}^{k^{2}}  \simeq \underset{\mu \subset \lambda, |\mu| = a}\bigoplus \Delta(\mu) \otimes \Delta(\lambda/\mu) \\
& \simeq & \underset{\mu \subset \lambda, |\mu| = a}\bigoplus \Delta(\mu) \otimes \bigg( \underset {\nu \subset \lambda, |\nu| = b} 
\bigoplus c^{\lambda}_{\mu,\nu} \Delta(\nu) \bigg) \\
& \simeq &  \underset{|\mu|=a,|\nu|=b,|\alpha|=d}\bigoplus c^{\lambda}_{\mu,\nu} \ c^{\alpha}_{\mu,\nu} \Delta(\alpha).
\end{eqnarray*}
We also have $\Gamma^{(a,b)} \simeq \Gamma^{(a+1,b-1)}\oplus \Delta((a,b))$ by, e.g., Pieri's formula.
Apply $-\intP \Delta(\lambda)$ and use the above calculation to get
\begin{center}
$\Delta(\lambda) \intP \Delta({(a,b))} \simeq 
\underset{\alpha}\bigoplus \bigg( \underset{|\mu|=a,|\nu|=b} \Sigma c^{\lambda}_{\mu,\nu} \ c^{\alpha}_{\mu,\nu} -
\underset{|\bar \mu|=a+1,|\bar \nu|=b-1} \Sigma c^{\lambda}_{\bar\mu,\bar\nu} \ c^{\alpha}_{\bar\mu,\bar\nu} \bigg)
\Delta(\alpha).$
\end{center}
As a special case consider $b=1$ and let $c=$ the number of outer corners of $\lambda$. 
Now $\Gamma^{(a+1,b-1)}= \Gamma^d$, which is the identity for $\intP$. We get 
\begin{center}
$\Delta(\lambda) \intP \Delta((a,1)) \simeq (c-1) \Delta(\lambda) \bigoplus \underset{\alpha}\bigoplus \Delta(\alpha),$ 
\end{center}
where $\alpha$ ranges over partitions obtained by moving exactly one box in the Young diagram of $\lambda$ elsewhere. 
Note that if we apply the Schur functor to this, we get the well known formula for 
the Kronecker product of a Specht module with the standard module.

\end{example}

\begin{example} 
We calculate $\Delta(\lambda)\intP\Delta((p,1^{q}))$ where $p, q$ are positive integers with $d = p+q$. 
Straightforward imitation of earlier procedure would require us to involve $\Gamma^\nu$ where $\nu$ has several parts. 
Instead we use polynomial functors whose arguments are super-vector spaces (see~\cite{Axt13}). 
In this language $\Gamma^p (V) \otimes \wedge^q (V) \simeq \Gamma^{d,k^+ \oplus k^-}_{(p,q)}(V)$. 
On the right hand side, the  parametrization by $k^+ \oplus k^-$ and taking the $(p,q)$ weight space amounts to 
requiring $p$ letters from the argument $V$ to commute and remaining $q$ letters to anticommute. 
Similarly, using obvious terminology, 
$\Delta_+(\mu)(V) = \Delta(\mu)(V_+) = \Delta(\mu)(V)$ whereas $\Delta_-(\mu)(V) = \Delta(\mu)(V_-) \simeq \nabla(\mu')(V)$, 
which is $\Delta(\mu')(V)$ for $k$ a field of characteristic 0. 
Proceeding as before via a super-analogue of the relevant filtration, we have
\begin{eqnarray*}
\Delta(\lambda) \intP \bigg(\Gamma^p \otimes \wedge^q \bigg) & \simeq & 
\Delta(\lambda)^{k^+ \oplus k^-}_{(p,q)}
\simeq \underset{\mu \subset \lambda, |\mu| = p}\bigoplus \Delta_+(\mu) \otimes \Delta_-(\lambda/\mu) \\
& \simeq & \underset{\mu \subset \lambda, |\mu| = p}\bigoplus \Delta_+(\mu) \otimes \Delta_+(\lambda'/\mu') \\
& \simeq & \underset{\mu \subset \lambda, |\mu| = p}\bigoplus \Delta(\mu) \otimes \bigg( \underset {\nu \subset \lambda, |\nu| = q} \bigoplus c^{\lambda'}_{\mu',\nu} \Delta(\nu) \bigg) \\
& \simeq &  \underset{|\mu|=p,|\nu|=q,|\alpha|=d}\bigoplus c^{\lambda'}_{\mu',\nu} \ c^{\alpha}_{\mu,\nu} \Delta(\alpha).
\end{eqnarray*}
Now again by Pieri's rule, $\Gamma^p \otimes \wedge^q \simeq \Delta((p,1^{q})) \oplus \Delta((p+1,1^{q-1}))$. 
This allows one to calculate $\Delta(\lambda) \intP\Delta(p,1^{q})$ as the alternating sum of 
$\Delta(\lambda) \intP \big(\Gamma^{p+i} \otimes \wedge^{q-i} \big)$ with $i = 0, \ldots, q$.

\end{example}


\end{document}